\documentclass{article}

\usepackage{latexsym}
\usepackage{amssymb}
\usepackage{amsmath}
\usepackage{enumerate}
\usepackage{amsmath,amsthm,amsfonts,amssymb,latexsym,xy}
\xyoption{all}

\usepackage{latexsym}
\usepackage{amsthm}
\usepackage{amssymb}
\usepackage{amsmath}
\usepackage{amsfonts}
\usepackage{mathrsfs}
\usepackage{tikz}
\usetikzlibrary{matrix}
\usepackage[T1]{fontenc}
\usepackage[utf8]{inputenc}
\usepackage{authblk}

\theoremstyle{definition}

\newtheorem{theorem}{Theorem}
\newtheorem{lemma}[theorem]{Lemma}

\newtheorem{corollary}[theorem]{Corollary}
\newtheorem{definition}[theorem]{Definition}

\newtheorem{claim}[theorem]{Claim}

\newcommand{\paths}{\Omega}

\newcommand{\aut}{\mathop{\mathrm{Aut}}}

\begin{document}

\title{Non-elementary amenable subgroups of automata groups}

\author{Kate Juschenko\thanks{kate.juschenko@gmail.com}}
\affil{Department of Mathematics, Northwestern University, USA}

\date{\today}

\maketitle

\begin{abstract}
We consider groups of automorphisms of locally finite trees, and give conditions on its subgroups that imply that they are not elementary amenable. This covers all known examples of groups that are not elementary amenable and act on the trees: groups of intermediate growths and Basilica group, by giving a more straightforward proof. Moreover, we deduce that all  finitely generated branch groups are not elementary amenable, which was conjectured by Grigorchuk. 
\end{abstract}
\section{Introduction}

A group $G$ is amenable if there exists a finitely additive translation invariant probability measure on all subsets of $G$. This definition was given by John von Neumann, \cite{von-neumann1}, in a response to Banach-Tarski, and Hausdorff paradoxes. He singled out the property of a group which forbids paradoxical actions.

The class of {\it elementary amenable groups}, denoted by  $EG$, was introduced by Mahlon Day in \cite{day:semigroups}, as the smallest class of groups that contain finite and abelian groups and is closed under taking subgroups, quotients, extensions and directed unions. The fact that the class of amenable groups is closed under these operations was already known to von Neumann, \cite{von-neumann1}, who noted at that at that time there was no known amenable group which did not belong to $EG$. 

We distinguish between "the group is not elementary amenable" and "the group is non-elementary amenable". The first saying stands for the class of all groups which are not in $EG$, they can be either amenable or not, or amenability is an open question for them. At the same time "non-elementary amenable" is the class of amenable groups which are not in $EG$.

No substantial progress in understanding this class has been made until~80s, when  Chou, \cite{C}, showed that all elementary amenable groups have either polynomial or exponential growth, and Rostislav Grigorchuk, \cite{grigorchuk:milnor_en} gave an example of a group with intermediate growth. Grigorchuk's group served as a starting point in developing the theory of groups with intermediate growth, all of them being non-elementary amenable. In the same paper Chou showed that every simple finitely generated infinite group is not elementary amenable. In \cite{JM} it was shown that the  topological full group of Cantor minimal system is amenable.  By the results of  Matui, \cite{Matui}, this group has a simple and finitely generated commutator subgroup, in particular, it is not elementary amenable. This was the first example of infinite simple finitely generated amenable group.

An important class of groups containing many non-elementary amenable groups is automata groups.
The first example of non-elementary amenable group with exponential growth is Basilica group introduced by Grigorchuk and Zuk, \cite{zukgrigorchuk:3st},  amenability of which was a technically difficult problem solved by  Bartholdi and Vir\'ag in \cite{barthvirag}. In \cite{bkn:amenability},  Bartholdi,  Nekrashevych and  Kaimanovich showed amenability of the group of bounded automata of finite state. This result was extended to automata of linear activity by Amir, Angel and Vir\'ag, \cite{amirangelvirag:linear}. In it's turn, this result was extended to automata with quadratic activity, \cite{JNS}. Moreover, a general approach for proving amenability of all known non-elementary amenable groups was introduced in \cite{JNS}. It is based on extensive amenability property, which was studied before in \cite{JM}, \cite{JS} and applied to interval exchange transformation group in \cite{JMMS}. 

While there is a unique approach for proving amenability of all currently known non-elementary amenable groups, there is no unique approach for proving that those groups are non-elementary amenable. Currently there are two different sources of non-elementary amenable groups: one comes from groups that act on trees and another is the class of topological full groups of Cantor minimal systems. In Theorem \ref{main} we give conditions on groups acting on trees which imply that they are not elementary amenable. The class of groups that satisfy this theorem contains the class of all known cases of groups that are not in $EG$ and act on a tree. In \cite{grig-review}, Grigorchuk conjectured that finitely generated branch groups are not elementary amenable. In \cite{MN}, Myropolska and Nagnibeda have obtained a partial result on this conjecture assuming in addition that the group is just infinite. As one of the applications of Theorem \ref{main} we obtain Grigorchuk's conjecture. We refer the reader to \cite{BGS} for an excellent survey on branch groups. We also present a direct proofs that Basilica and Grigorchuk's group are not elementary amenable, which we believe are quite illustrative. It also follows that the iterated monodromy group of quadratic polynomials (except $z^2-2$  and $z^2$) discussed in \cite{BN} are not elementary amenable.

\bigskip

{\bf Acknowledgments:} I am grateful to Rostyslav Grigorchuk for motivating us to enter the subject, and sharing ideas and insights around Basilica group. Volodia Nekrashevych made an important observation on the early draft of the paper, which helped us to prove that the class of groups which are not elementary amenable is much larger then we expected.

\section{Automorphisms of homogeneous trees}
Let $\Omega=(E_{1}, E_{2},
\ldots)$ be a sequence of finite sets.
 Define $\paths_n=E_1\times E_2\times\cdots\times E_n$,  where $\paths_0$ is a
singleton. The disjoint union
$\paths^*=\bigsqcup_{n\ge 0}\paths_n$ can be naturally identified with vertices of  a rooted tree. The sets
$\paths_n$ are its levels, i.e., we will  consider them as the sets of vertices on the $n$-th level of the rooted tree. Vertices $v_1\in\paths_n$ and
$v_2\in\paths_{n+1}$ are connected by an edge if and only if $v_2$ is
a continuation of $v_1$, namely $v_2=v_1e$ for some $e\in
E_{n+1}$. 

Denote
by $\aut(\paths^*)$ the automorphism group of the rooted tree $\paths^*$.
The group $\aut(\paths^*)$ acts transitively
on each of the levels $\paths_n$. Denote by $_n\paths^*$ the tree of finite paths of the
``truncated'' diagram defined by the sequence $$_m\paths=(E_{m+1}, E_{m+2},
\ldots).$$

For every $g\in\aut(\paths^*)$ and $v\in\paths_n$ there exists an
automorphism $g|_v\in\aut(_m \paths^*)$ such that
\[g(vw)=g(v)g|_v(w)\]
for all $w\in$ $_m\paths^*$. The automorphism $g|_v$ is called the
\emph{section} of $g$ at  $v$.

We list obvious properties of sections. For all $g_1, g_2, g\in\aut(\paths^*)$, $v\in\paths^*$,
$v_2\in$$_n \paths^*$ and  $v_1\in\paths_n$ we have
\begin{equation*}
(g_1g_2)|_v=g_1|_{g_2(v)}g_2|_v,\qquad g|_{v_1v_2}=g|_{v_1}|_{v_2}.
\end{equation*}

An automorphism $g\in\aut\Omega^*$ is said to be \emph{finitary} of
\emph{depth} at most $n$ if all sections $g|_v$ for $v\in\Omega_n$ are trivial.

Let $g\in\aut\Omega^*$. Denote by $\alpha_n(g)$ the number of paths
$v\in\Omega_n$ such that $g|_v$ is non-trivial. We say that
$g\in\aut\Omega^*$ is \emph{bounded} if the sequence $\alpha_n(g)$ is bounded.\\

If $g\in\aut\Omega^*$ is bounded, then there exists a finite set $P\subset\paths$
of infinite paths such that $g|_v$ is non-finitary only if $v$ is a
beginning of some element of $P$. Consequently, bounded automorphisms
of $\Omega^*$ act on $\Omega$ by homeomorphisms of bounded type.\\

Suppose that the sequence $\Omega=(E_1, E_2, \ldots)=(X, X, \ldots)$ is
constant, so that $_m\Omega^*$ does not depend on $m$, and $X=\{x_1,\ldots,x_d\}$. In this case, we will denote $\Omega^*$ by $X^*$.
An automorphism $g\in\aut\Omega^*$ is said to be
\emph{finite-state} if the set
$\{g|_v\;:\;v\in\Omega^*\}\subset\aut\Omega^*$ is finite.

Let $Aut(X^*)^X$ be the direct product of $d$ copies of $Aut(X^*)$. Consider the wreath product $$Aut(X^*)\wr S_d=Aut(X^*)^X\rtimes S_d,$$
with the multiplication  is given by 
$$(a_1,\ldots, a_d)\varepsilon\cdot (b_1,\ldots, b_d)\nu=(a_1b_{\varepsilon^{-1}(1)},\ldots, a_db_{\varepsilon^{-1}(d)})\varepsilon\nu,$$
for every $(a_1,\ldots, a_d)\varepsilon, (b_1,\ldots, b_d)\nu\in Aut(X^*)^X\rtimes S_d$. Note that there exists a homomorphism from $Aut(X^*)$ into $Aut(X^*)\wr S_d$, given by $$g\mapsto (g|_{x_1},g|_{x_2},\ldots, g|_{x_d})\sigma_g,$$
where $\sigma_g$ is the permutation that $g$ induces on the first level of the tree. In the later sections, we will identify $g$ with its image under this map.

Assume that $G$ fixes the $n$-th level of the tree, then we can define a projection on the tree that grows from a vertex on the $n$-th level.

We will use the following notations. The stabilizer of a vertex $v$ of the tree will be denoted by $Stab_G(v)$. The stabilizer of the $n$-th level is denoted by $Stab_G(n)$. 

{\it The rigid stabilizer} of a vertex $v\in _n\paths^*$ is defined by
$$rist_G(v)=\{g\in G: g\in St_G(n), \text{ if } w\neq v, w\in _n\paths^* \text{ then }g|_{w}=id \}.$$

The rigid of the $n$-th level, $rist_G(n)$, is defined as a group generated by all $rist_G(v)$ with $v\in _n\paths^*$. In other words, $$rist_G(n)=\prod\limits_{v\in  _n\paths^*} rist_G(v).$$

\section{Elementary amenable groups}
Let $EG_0$ be the class of all finite and abelian groups. Assume that $\alpha$ is an ordinal such that for all ordinals $\beta<\alpha$ we have already constructed the class $EG_{\beta}$. If $\alpha$ is a limit ordinal we set $$EG_{\alpha}=\bigcup\limits_{\beta<\alpha} EG_{\beta}.$$

Otherwise we set $EG_{\alpha}$ as the class of groups which can be obtained from $EG_{\alpha-1}$ by applying one time either extension or direct union.

It was shown by Chou, \cite{C}, that for every ordinal $\alpha$ the class $EG_{\alpha}$ is closed under taking subgroups and quotients. In the same paper, he showed that the class of elementary amenable groups is the smallest class which contains all finite and all abelian groups and is closed under taking extensions and direct limits.

Similarly to elementary amenable groups one can define a class of elementary subexponetially amenable groups. Recall that a group $\Gamma$ generated by a finite set $S$ is of subexponential growth if $\lim_n |B(n)|^{1/n}=1$, where $B(n)$ is the ball of radius $n$ in the Cayley graph $(\Gamma,S)$. 

Let $SG_0$ be the class of groups whose all finitely generated subgroups have subexponential growth. Let $\alpha>0$ be an ordinal such that we have already defined $SG_{\beta}$ for all $\beta<\alpha$. Then if $\alpha$ is a limit ordinal, we set 
$$SG_{\alpha}=\bigcup\limits_{\beta<\alpha} SG_{\beta}$$
and if $\alpha$ is not a limit ordinal, we set $SG_{\alpha}$ to be the class of groups which can be obtained from $SG_{\alpha-1}$ by applying either extension or a direct union one time. Define
$$SG=\bigcup\limits_{\alpha} SG_{\alpha}$$

Following Chou's ideas one can show that the class $SG$ is the smallest class of groups which contain all groups of sub-exponential growth and is closed under taking subgroups, quotients, extensions, and direct unions. Moreover, each $SG_{\alpha}$ is closed under taking subgroups and quotients.

The class $SG$ will be called  {\it the class of elementary subexponentially amenable groups}. Obviously, the class $SG$ contains the class of elementary amenable groups $EG$.\\

In the following theorem we give a condition on finitely generated subgroups of  automorphisms of trees, that imply that these groups are not in the class $EG$. We will discuss amenability of these groups later.

\begin{theorem}
Let $\mathcal{A}$ be a class of finitely generated, non-abelian, infinite groups acting on locally finite rooted trees. Assume that for every $G$ in $\mathcal{A}$, for every vertex $v$ the rigid stabilizer  $rist_G(v)$ contains a subgroup  admitting a homomorphic image in $\mathcal{A}$. Then every group of $\mathcal{A}$ is not elementary amenable. Moreover, if all groups of $\mathcal{A}$ are of exponential growth, then the groups of $\mathcal{A}$ are not elementary subexponentially  amenable. 
\end{theorem}
\begin{proof}
Let us firstly deduce the statement from the following claim and then prove the claim itself.\\

\begin{claim}\label{claim} Let $G\in \mathcal{A}$. For every normal subgroup $H$ of $G$, there exists  a vertex $v$ a subgroup $H'$ of $H$ and a homomorphism $\pi:H'\rightarrow H''$ onto some group $H''$ which contains $K:=rist_G(v)$ as a subgroup.\\
\end{claim}

Firstly, we will deduce the statement from the claim. Let $G\in \mathcal{A}$. Since $G$ is infinite and non-abelian, we have that it is not in the class $EG_0$. To reach a contradiction, assume $G\in EG_{\alpha}$, where $\alpha$ is  a minimal ordinal among all such that $G\in EG_{\alpha}$ and $G\in \mathcal{A}$.  Clearly, $\alpha$ is  not limit ordinal. Thus $G$ is obtained from the class $EG_{\alpha-1}$ by taking either extension or a direct union. It can not be obtained as a direct union, since otherwise $G\in EG_{\alpha-1}$ which contradicts  minimality of $\alpha$. Hence $G$ is obtained as an extension of groups from the class $EG_{\alpha-1}$.  Therefore we can find a normal subgroup  $H$ of $G$, which is in the class $EG_{\alpha-1}$. 
Then $H'$, $H''$, and, thus $K$, are in $EG_{\alpha-1}$. Since $K$ contains a subgroup  (again in $EG_{\alpha-1}$) admitting a homomorphic image in $\mathcal{A}$, which implies that there exists a group in $\mathcal{A}\cap EG_{\alpha}$. This contradicts minimality of $\alpha$.  Since the class of subexponetially amenable groups is closed under taking subgroups and quotients, exactly the same considerations imply that if $G$ is of exponential growth then it is not subexponetially elementary amenable.

To prove the claim we will consider the wreath product presentation of $G$. Let $g$ be a non-trivial element in $H$. Then let $n$ be a level of the tree which is fixed by $g$ and the next level is not fixed by $g$. Then, for some $d$ and $1\leq k\leq d$, $g$ has the form
$$g=(g_1, \ldots, g_{k-1},\overline{g},g_{k+1},\ldots, g_{d})$$
where $\overline{g}$ is an automorphism of the tree which does not fix the first level. In other words, $$\overline{g}=(\overline{g}_1,\overline{g}_2,..,\overline{g}_d)s,$$ for some $g_1,g_2,..,g_d \in G$ and a permutation $s\in Sym(d)$, where $d$ corresponds to the number of vertices of the $n$-th level.

Consider an automorphism of the tree of the form 
$$h=(h_1, h_2,\ldots, h_{d}),$$ 
for some automorphisms  $h_i\in K$. 
Then 
$$[g,h]=([g_{1},h_1],\ldots,[\overline{g}, h_{k}],\ldots, [g_{d}, h_{d}])\in H.$$

By assumptions $\overline{h}=h_{k}$ can be chosen in $K\times K\times \ldots \times K$. Assume $m$ is not fixed by the permutation $s$, and let  $\overline{h}=(1,\ldots, 1, w, 1,\ldots 1)$, where $w\in K$ is placed on the $m$'s position.  Hence, the $m$-th coordinate of $[\overline{g},\overline{h}]$ is equal to $w$. Thus, $H$ admits a subgroup, whose homomorphic image contains $K$ as a subgroup. \\
\end{proof}

The statement of the next corollary follows follows from assumption $\mathcal{A}=\{G\}$. 

\begin{corollary}\label{main}
Let $G$ be a finitely generated, non-abelian, infinite group of automorphisms of a $d$-homogeneous tree $\mathcal{T}$. If there exists a subgroup $K<G$ such that the following holds:
\begin{enumerate}[(i)]
\item $K\times K\times \ldots \times K<K$, where the product is taken $d$ times,
\item for some vertex $v$ on the tree, $p_v(Stab_{K}(v))$ contains $G$.
\end{enumerate}
Then $G$ is not elementary amenable. Moreover, if $G$ is of exponential growth, then $G$ is not subexponentially elementary amenable. 
\end{corollary}

\section{Applications}

\paragraph*{Branch groups.} 
Suppose that $m=\{m_i\,|\,i\geq 0\}$ is a sequence of natural numbers. We denote by $T_m$ the rooted tree, such that that any vertex at level $i$ has $m_i$ children. Such trees are called spherically homogeneous.\\

For a vertex $v$ of $T_m$ denote $T_v$ the subtree of $T_m$ which lies under $v$.  Let $|v|$ be the distance from the root to $v$.

A group $G$ acting on a spherically homogeneous tree is called {\it branch group} if the following conditions are satisfied:
\begin{enumerate}
\item All rigid stabilizers are of finite index in $G$, i.e., $|G:rist_G(n)|<\infty$ for every $n\in \mathbb{N}$;
\item $G$ acts transitively on each level.
\end{enumerate}

We remark that branch groups can not be of polynomial growth. We start with analyzing rigid stabilizers.

\begin{definition}
Let $G$ be a group of automorphisms on $T_m$, and $v$ a vertex of $T_m$. Denote by $st_G(v)$ the subgroup of $G$ that fixes $v$. For $g\in st_G(v)$ denote $g_{|v}=\pi_v(g)$ the automorphism of the subtree $T_v$, obtained by restricting $g$ to $T_v$.
\end{definition}
If $G$ is a branch group, we are going to show that every $rist_G(v)$ has a homomorphic image that is also branch. To construct it, we will choose a subgraph of $T_v$ where $rist_G(v)$ acts transitively, and restrict $rist_G(v)$ to that subgraph.\\

We start with a lemma. Let $Y$ be a subset of vertices of level $n$ of the tree $T_m$. Denote by $T_Y$ the union of subtrees $T_v$, $v\in Y$. If $g$ fixes each element of $Y$, denote by $g_{|Y}$ the restriction of $g$ to $T_Y$.
\begin{lemma}\label{restriction_index}
  Let $G$ be a branch group. Then $\prod_{v\in Y}rist_G(v)$ is of finite index in $st_G(n)_{|Y}$. 
\end{lemma}
\begin{proof}
Since $G$ is branch, $\prod_{v:|v|=n}rist_G(v)$ is of finite index in $st_G(n)$. Restricting to $Y$, we have that $(\prod_{v:|v|=n}rist_G(v))_{|Y}$ is of finite index in $st_G(n)_{|Y}$. It is left to notice that restricting to $Y$ kills every $rist_G(v)$ for $v\not\in Y$, and so $(\prod_{v:|v|=n}rist_G(v))_{|Y}=\prod_{v\in Y}rist_G(v)$
\end{proof}
We  have the following corollary. For every $v$, $rist_G(v)$ is of finite index in $st_G(v)_{|v}$. To derive it from the lemma, take $Y=\{v\}$.\\

We also need the following general lemma
\begin{lemma}
Suppose $G$ acts transitively on a set $X$, and $H<G$ is of index at most $k$. Then the number of $H$-orbits in $X$ is at most $k$. 
\end{lemma}
\begin{proof}
Let $G=\cup_{i=1}^k Hg_i$. Then $X=Gx=\cup_{i=1}^k Hg_i x$, where each $Hg_i x$ is an $H$-orbit.
\end{proof}

\begin{lemma}\label{rist}
If $G$ is a branch group, for every vertex $v$, $rist_G(v)$ has a homomorphic image that is also branch.
\end{lemma}
\begin{proof}
Since $G$ is transitive on levels of the tree, it follows that $st_G(v)$ is transitive on levels of the subtree $T_v$. The group $rist_G(v)$ may not be transitive on levels of $T_v$, but the number of orbits on each level is uniformly bounded, since $rist_G(v)$ is of finite index in $st_G(v)_{|v}$ by Lemma \ref{restriction_index}.

Let $O_n$ be the set of orbits of $rist_G(v)$ on the $n$-th level of $T_v$. If $Y\in O_n$ is an orbit, define $e(Y)\in O_{n-1}$ as the orbit of (some) parent of a vertex in $Y$. Since $rist_G(v)$ consists of tree automorphisms, the map $e:O_n\to O_{n-1}$ is well-defined (i.e. does not depend on the choice of a vertex in $Y$). We should actually use $e_n$, by we suppress the dependence on $n$ to ease the notation.

Now note that $e$ is surjective for all $n$, since each vertex has children. It follows that the number of elements in $O_n$ cannot decrease. Since it is uniformly bounded, it must become constant after some $n$, and after that the map $e:O_n\to O_{n-1}$ is a bijection, since it is a surjective map between sets of equal size.

It follows that for some $n_0$ there is are orbits $Y_n\in O_n$ for $n\geq n_0$, such that $e^{-1}(Y_n)=Y_{n+1}$, that is, that all children of vertices in $Y_n$ lie in $Y_{n+1}$.  Let $\Gamma$ be the graph which is the union of subtrees $T_w$, $w\in Y_{n_0}$. Then $\Gamma$ is invariant with respect to $rist_G(v)$. Note that $Y_n$ is the union of level sets of $T_w$, $w\in Y_{n_0}$, and that recall that $rist_G(v)$ acts transitively on $Y_n$ for each $n\geq n_0$.

By Lemma \ref{restriction_index} the subgroup $\prod_{w\in Y_n}rist_G(w)$ is of finite index in $rist_G(v)_{|\Gamma}$.

It is left to note that by connecting all vertices in $Y_{n_0}$ to a new vertex (root), we turn $\Gamma$ into a tree on which $rist_G(v)$ acts. By above this action is transitive on levels, and the rigid stabilizers of the levels are of finite index. Thus $rist_G(v)_{|\Gamma}$ is a branch group.
\end{proof}

\begin{corollary}
Finitely generated branched groups are not elementary amenable.
\end{corollary}
\begin{proof}
Let $\mathcal{A}$ be the class of all finitely generated branch groups. Then for every for a  vertex $v$ there exist a group from $\mathcal{A}$ on which $rist_G(v)$ projects. Indeed, for a given group $G$ in $\mathcal{A}$, we have that rigid stabilizer of each level is finitely generated. Since it is a product of rigid stabilizers of the vertices on the $n$-th level, we have that all rigid stabilizers of any vertex are finitely generated. By Lemma \ref{rist}, we have that each rigid stabilizer projects onto a branch group, thus, a group from $\mathcal{A}$. Therefore Theorem \ref{main} implies the statement.
\end{proof}

In the next examples we give proofs from scratch, instead of proving that the groups are branch.

\paragraph*{Basilica group.} For the application we will only consider the wreath product presentation of Basilica group, which was introduced in  \cite{zukgrigorchuk:3st}: 
$$a=(1,b) \text{ and } b=(1,a)\varepsilon,$$
where $\varepsilon$ is a non-trivial element of the group $\mathbb{Z}/2\mathbb{Z}$. The proof of the following becomes almost straightforward. 
\begin{corollary}
Basilica group is not elementary subexponentially amenable. In particular, it is not elementary amenable. 
\end{corollary}
\begin{proof}Let $G$ be Basilica group.
Firstly we note that, by \ref{subexponential}, $G$ contains a free semigroup and, thus, it is of exponential growth. We will apply Theorem \ref{main} with $K=G'$.

{\it Claim:} $G'\geq G'\times G'$. 
We start by showing that the projections onto both coordinates of $St_G(1)$ coincide with $G$. Direct computation shows
\begin{align*}
a^b&=(a^{-1},1)\varepsilon (1,b)(1,a)\varepsilon =(a^{-1}ba,1)=(b^a,1)\in St_G(1)\\
b^2&=(a,a)\in St_G(1),
\end{align*}
and moreover $a=(1,b)\in St_G(1)$. Thus, $p_1(St_G(1))=p_2(St_G(1))=G$.\\ 
Note that 
$$[a,b^2]=(1,b^{-1})(a^{-1},a^{-1})(1,b)(a,a)=(1,b^{-1}a^{-1}ba)=(1,[b,a]).$$
Summarizing above, we obtain
$$G'\geq \langle [a,b^2]\rangle^G\geq  \{e\}\times \langle [b,a] \rangle^G=\{e\}\times G'.$$
Moreover, $(\{e\}\times G')^b=G'\times \{e\}$, therefore the claim follows.

Since we have the following expressions
$$[b^{-1},a]=(b^{-1},b)\text{ and }[a,b^2]=(1,b^{-1}a^{-1}ba)$$
we obtain that $b, b^a \in p_2(G')$. But
$$b^2=(a,a) \text{ and }b^a=(b,b^{-1}a).$$
Thus $p_2(p_2(G'))$ contains $G$. In \cite{zukgrigorchuk:3st}, it was shown that $G$ is of exponential growth, we add the proof of this in the  Lemma \ref{subexponential} for completeness. Applying Theorem \ref{main} we obtain the  statement.
\end{proof}

\begin{lemma}[Grigorchuk-Zuk, \cite{zukgrigorchuk:3st}]\label{subexponential}
The semigroup generated by $a$ and $b$ is the free non-abelian group $\mathbb{F}_2$.
\end{lemma}
\begin{proof}

To reach a contradiction consider two different words $V$ and $W$  that represents the same element in $G$ and such that $\rho=|V|+|W|$ is minimal. It is easy to check that $\rho$ can not be $0$ or $1$.

From the wreath product representation of $G$ we see that the parity of occurrences of $b$ in $V$ and in $W$ should be the same. Both words $V$ and $W$ are products of the elements of the form
$$a^n=(1,b^n), \quad n\geq 0$$
or of the form
$$ba^m b=(1,a)\varepsilon  (1,b^m)(1,a)\varepsilon=(b^ma,a), \quad m\geq 0.$$
Assume that one of the words does not contain $b$, say $W=a^n$. Then, since $a$ is of infinite order, $V$ must contain $b$. Projecting $V$ to the first coordinate we obtain a word which is equal to identity. Moreover the length of this word is strictly less $\rho$, which gives a contradiction. Thus we can assume that both words contain $b$. 

Suppose that the number of $b$'s in $W$ and $V$ is one, then by minimality these words should have the form $ba^n$ and $a^mb$.  But $ba^n=(b^n,a)\varepsilon$ and $a^mb=(1,b^ma)\varepsilon$, which represent two different words.

Multiplying by $b$ both words $V$ and $W$, we can assume that $b$ appears even number of times. This can increase the length of  $V$ and $W$  by at most $1$. Suppose that in both words $b$ appears twice. 

Thus either both words contain two $b$'s and $|V|+|W|= \rho$ or one of them contains at least four $b$'s, in which case we have $|U|+|V|\leq \rho+2$. In the first case, taking the projection of the words onto the second coordinate we see that $\rho$ decreases by $2$ for these projections, which contradicts  minimality. Similarly, in the second case $\rho$ decreases by $3$ which again contradicts minimality. Hence, the statement follows.
\end{proof}

\paragraph*{The first Grigorchuk's group.}  Historically the first Grigorchuk's group was defined in \cite{grigorchuk:80_en}, as a group  generated by $4$ automorphisms of the binary tree. Its  wreath product presentation  is the following:
$$a=(1,1)\varepsilon,\quad b=(a,c), \quad c=(a,d), \quad d=(1,b),$$
where $\varepsilon$ is a non-trivial element of the group $\mathbb{Z}/2\mathbb{Z}$.
\begin{corollary}
Grigorchuk's group  is not elementary amenable. 
\end{corollary}
\begin{proof}
Let $K$ be a normal subgroup generated by $(ab)^2$. It is well known, see for example \cite{Harpe}, that $K\times K< K$. We will prove this here for completeness. Note that $a^2=b^2=c^2=d^2=1$, $b=cd$ and $b,c,d$ generate abelian group. It is easy to check that both $p_0(St_G(1))$ and $p_1(St_G(1))$ contain $G$. Since $(abad)^2=(ab)^2d^{-1}(ab)^{-2}d$, we have $(abad)^2$ is in $K$. It follows 
$$(abad)^2=(1,abab)\in K.$$
Moreover, $a(abad)^2a^{-1}=(abab,1)\in K$. Therefore, we have $K\times K<K$.\\

Now we will show that there exists a vertex $v$ such that $p_1(St_K(v))$ contains $G$. Note that 
$$(ab)^2=(ca,ac)\in K.$$
Therefore, $p_1(St_K(1))$ contains $ca$ and, thus, it contains $H=\langle ca \rangle^G$. We claim that $ac,ad,b$ are in $H'=p_1(St_H(1))$. This implies that $\langle ad, b \rangle^G$ is a subgroup in $H'$. From this it follows that the elements $b=(a,c)$, $b^{a}=(c,a)$, $(ad)b(ad)=(d,ab)$ are in $H'$, and $p_1(St_{H'}(1))$ contains $G$ (because it contains $a,c$ and $d$), which implies the statement. 

Our claim follow from the following computations: 
\begin{align*}
ca&=(a,d)\varepsilon,  \\
(ca)^2&=(ad,da)\in H, \\
ac(ca)^d&=(d,a)\varepsilon (ab,bd)\varepsilon=(b,b)\in  H.
\end{align*}
\end{proof}

\end{document}